\documentclass[12pt]{amsart}

\newcommand{\ilimit}{\mbox{$\,\displaystyle{\lim_{\longleftarrow}}\,$}}

\usepackage[arrow,matrix,curve]{xy}\SilentMatrices
\def\xyma{\xymatrix@M.7em}

\newcommand{\para}{\par\vspace{.25cm}}
\newtheorem{Lemma}{Lemma}[section]

\newtheorem{Theorem}[Lemma]{Theorem}

\newtheorem{prop}[Lemma]{Proposition}
\newtheorem{Conjecture}[Lemma]{Conjecture}

\usepackage{amssymb}
\usepackage{amsfonts}
\usepackage{euscript}
\begin{document}

\title{Limits over categories of extensions}
\author{Roman Mikhailov and Inder Bir S. Passi}
\begin{abstract}
We consider limits over categories of extensions and show how certain well-known
functors on the category  of groups  turn out as such limits. We also discuss
higher (or derived) limits over categories of extensions.
\end{abstract}
\maketitle
\section{Introduction}
Let $k$ be a commutative ring with identity, and let $\EuScript C$
be one of the following categories: the category $\sf Gr$ of
groups, the category $\sf Ab$ of abelian groups,  the category ${\sf Ass}_k$
of associative algebras over $k$. Given an object $G\in\EuScript
C$,  let  ${\sf Ext_\EuScript C}(G)$ be the category whose
objects are  the extensions $H\rightarrowtail F\twoheadrightarrow
G$ in $\EuScript C$ with $G$ as the cokernel, and the morphisms
are the commutative diagrams of short exact sequences of the form
$$
\xyma{
H_1 \ar@{^{(}->}[r] \ar@{->}[d] & F_1 \ar@{->>}[r] \ar@{->}[d] & G \ar@{=}[d]\\
H_2 \ar@{^{(}->}[r] & F_2 \ar@{->>}[r] & G. }
$$
It is clearly natural
to consider also the full subcategory ${\sf Fext_\EuScript C}(G)$ of ${\sf
Ext_\EuScript C}(G)$ which consists of the short exact sequences $H\rightarrowtail F\twoheadrightarrow  G$
where $F$ is a free object in $\EuScript C$.
The category ${\sf Ext_{Gr}}(G)$ has been
studied extensively  from the point of view of the theory of  cohomology of groups (see, for example,  \cite{GR:77}, \cite{GR:80}).   The general question which we wish to address here can be formulated as follows: how can one
study the properties of objects $G\in \sf Ob(\EuScript C)$ from the
 properties of the category ${\sf Fext_\EuScript C}(G)$?
\para
Let $\sf C$ be a small category and ${\mathfrak F} : \sf C
\longrightarrow \EuScript C$ a covariant functor. The inverse limit
$\ilimit {\mathfrak F}$ of ${\mathfrak F}$, by definition,  consists of those families $(x_{c})_{c\in \sf C}$ in the
direct product $\prod_{c \in \sf C} {\mathfrak F}(c)$  which are compatible in
the following sense: For any two objects $c,\,c' \in \sf C$ and any
morphism $a \in \mbox{Hom}_{\sf C}(c,\,c')$, we have ${\mathfrak
F}(a)(x_{c}) = x_{c'} \in {\mathfrak F}(c')$.
Let ${\mathfrak F},\,{\mathfrak G}$ be two functors from $\sf C$ to the category
$\EuScript C$. Then a natural transformation $\eta : {\mathfrak F}
\longrightarrow {\mathfrak G}$ induces a homomorphism
\[ \ilimit \eta : \ilimit {\mathfrak F} \longrightarrow \ilimit {\mathfrak G} , \]
by mapping any element $(x_{c})_{c\in\sf C} \in \ilimit {\mathfrak
F}$ onto $(\eta_{c}(x_{c}))_{c\in \sf C} \in \ilimit {\mathfrak
G}$. In this way, $\ilimit$ itself becomes a functor from the
functor category ${\EuScript C}^{\sf C}$ to the category $\EuScript C$.
\para
Our aim in this paper is to consider the categories ${\sf Ext}_{\EuScript C}(G), \ {\sf
Fext_\EuScript C}(G)$   and limits $\ilimit
\mathfrak F$ for functors
$
\mathfrak F_G: {\sf Ext_\EuScript C}(G)\to {\sf Gr}, \ \mathfrak F_G: {\sf
Fext_\EuScript C}(G)\to {\sf Gr}.
$
Suppose these functors are natural in the following sense:\para
Given a morphism $\alpha: G_1\to G_2$ in $\EuScript C$, every
commutative diagram of the form
$$
\xyma{
H_1 \ar@{^{(}->}[r] \ar@{->}[d] & F_1 \ar@{->>}[r] \ar@{->}[d] & G_1 \ar@{->}[d]^\alpha\\
H_2 \ar@{^{(}->}[r] & F_2 \ar@{->>}[r] & G_2 }
$$
induces a homomorphism of groups
$$
\mathfrak F_{G_1}\{H_1\rightarrowtail F_1\twoheadrightarrow G_1\}\to \mathfrak
F_{G_2}\{H_2\rightarrowtail F_2\twoheadrightarrow G_2\},
$$
compatible with morphisms in ${\sf Ext_\EuScript C}(G_1)$ and ${\sf
Ext_\EuScript C}(G_2)$, i.e., every commutative diagram of the form
\begin{equation}\label{dz1}{\small { \xyma{& & H_1' \ar@{^{(}->}[rr] \ar@{->}[ld]
\ar@{-}[d] && F_1' \ar@{-}[d] \ar@{->}[ld] \ar@{->>}[rr] && G_1 \ar@{->}[ld]_\alpha \ar@{=}[dd]\\
& H_2' \ar@{->}[dd] \ar@{^{(}->}[rr] & \ar@{->}[d] & F_2' \ar@{->}[dd] \ar@{->>}[rr] & \ar@{->}[d] &G_2 \ar@{=}[dd]\\
& & H_1 \ar@{^{(}-}[r] \ar@{->}[ld] & \ar@{->}[r] & F_1 \ar@{->}[ld] \ar@{-}[r] & \ar@{->>}[r] &G_1 \ar@{->}[ld]_\alpha\\
& H_2 \ar@{^{(}->}[rr] & & F_2  \ar@{->>}[rr] && G_2} }}
\end{equation}
induces the following commutative diagram of groups
\begin{equation}\label{dz2} \xyma{\mathfrak F_{G_1}\{H_1'\rightarrowtail F_1'\twoheadrightarrow G_1\}
\ar@{->}[r] \ar@{->}[d]&
\mathfrak F_{G_1}\{H_1\rightarrowtail F_1\twoheadrightarrow G_1\} \ar@{->}[d]\\
\mathfrak F_{G_2}\{H_2'\rightarrowtail F_2'\twoheadrightarrow
G_2\} \ar@{->}[r] & \mathfrak F_{G_2}\{H_2\rightarrowtail
F_2\twoheadrightarrow G_2\}.} \end{equation} In that case the
limit of the functors over categories of extensions defines a
functor ${\EuScript C\to \sf Gr}$ by setting $G\mapsto \ilimit
\mathfrak F_G$.
\para To clarify our point of view further, let us recall some known examples which have motivated our  present  investigation. Let ${\EuScript C=\sf Gr},$ $G$ a group, $\mathbb Z[G]$ its integral group ring and $M$ a $\mathbb
Z[G]$-module. For $n\geq 1$, define the functor
\begin{equation}\label{func1}
\mathfrak F_n: {\sf Fext_\EuScript C}(G)\to {\sf Ab}
\end{equation}
by setting $$\mathfrak F_n: \{ R\rightarrowtail F\twoheadrightarrow G\}\mapsto
(R_{ab})^{\otimes n}\otimes_{\mathbb Z[G]}M,$$ where $R_{ab}$ denotes the abelianization of $R$,  the action of
$G$ on $R_{ab}^{\otimes n}$ is diagonal and is defined via conjugation in $F$. It
is shown by I. Emmanouil and R. Mikhailov in \cite{EM} that $\ilimit \mathfrak F_n$ is isomorphic to the homology group $H_{2n}(G,\, M)$:
\begin{equation}\label{func11} \ilimit \mathfrak F_n\simeq H_{2n}(G,\, M).
\end{equation}
\para
For any group $G$ and  field $k$ of characteristic $0$ viewed as a trivial $G$-module,  the homology groups $H_n(G,\,k)$, $n\geq 2$,  appear as inverse limits for suitable natural functors defined on the category ${\sf Ext_{Gr}(G)} $ (see \cite{EP} for details).

\para
Consider the category ${\EuScript C}={\sf Ass}_{\mathbb Q}$ of
associative algebras over $\mathbb Q$, the field of rationals. For
integers $n\geq 1,$ and an associative algebra $A$ over $\mathbb
Q$, consider the functors
$$
\mathfrak F_n: {\sf Ext_\EuScript C}(A)\to \mathbb
Q\text{-modules}
$$
given by setting
\begin{equation}\label{func2}
\mathfrak F_n: \{I\rightarrowtail R\twoheadrightarrow A\}\mapsto R/(I^{n+1}+[R,\,R]),
\end{equation}
where $[R,\,R]$ is the $\mathbb Q$-submodule of $R$, generated by the
elements $rs-sr$ with  $r,\,s\in R$. It has been  shown by D. Quillen in \cite{Q}
that the inverse limit $\ilimit \mathfrak F_n$ is isomorphic to the even cyclic homology group $HC_{2n}(A,\,\mathbb Q)$:
\begin{equation}\label{func22}
HC_{2n}(A,\,\mathbb Q)\simeq \ilimit \mathfrak F_n.
\end{equation}
Furthermore, the Connes suspension map $S: HC_{2n}(A,\,\mathbb Q)\to
HC_{2n-2}(A,\,\mathbb Q)$ can be obtained as follows:
$$
\xyma{HC_{2n}(A,\,\mathbb Q) \ar@{=}[r] \ar@{->}[d]^S& \ilimit
\mathfrak F_n
\ar@{->}[d]\\
HC_{2n-2}(A,\,\mathbb Q) \ar@{=}[r] & \ilimit \mathfrak F_{n-1}}
$$
where the right hand side map is induced by the natural projection
$$
R/(I^{n+1}+[R,\,R])\to R/(I^n+[R,\,R]).
$$\para
The motivation for our investigation should now be  clear from the
above examples: one takes quite simple functors, like
(\ref{func1}) and (\ref{func2}), on the appropriate categories of
extensions and asks for the corresponding inverse limits.
It is then also natural to consider the derived functors
$$
\ilimit^i: {\sf Ab}^{{\sf Fext}(G)}\to \sf Ab
$$of the limit functor.
Every functor $\mathfrak F: {\sf Fext}_\EuScript C(G)\to \sf Ab$ which is
natural in the above-mentioned  sense (see diagrams (\ref{dz1}) and
(\ref{dz2})) determines a series of functors $\EuScript C\to \sf Ab$ by
setting $G\mapsto \ilimit^i \mathfrak F\in {\sf Ab},\ i\geq 0$.
\para

We now briefly describe the contents of the present paper.
We begin by recalling in Section 2  two properties of the limits given in
 \cite{EM} and \cite{EP}. The first (Lemma \ref{lemmaem}) provides a set of
 vanishing conditions for $\ilimit {\mathfrak F}$ while, the second states that $\ilimit {\mathfrak F}$
 embeds in $\mathfrak F(c_0)$ for every quasi-initial object $c_0$.  In Section 3  (Theorems 3.1 and 3.4)
 we show how the derived functors in the sense of A. Dold and D. Puppe \cite{DP} of certain standard non-additive functors
 on {\sf Ab}, like tensor power, symmetric power, exterior power, are realized as limits of suitable functors over extension categories.  In Section 4 we discuss higher limits and prove (Theorem 4.4) that for the functor
$$\mathfrak F:{\sf Fext}_{\sf Gr}(G)\to {\sf Ab},\quad (R\rightarrowtail F\twoheadrightarrow
G)\mapsto R/[F,\,R],$$ $\ilimit^1\mathfrak F$ is non-trivial if $G$ is not a perfect group.  We  conclude with some remarks and possibilities for further work in Section 5.
\para\vspace{.5cm}
\section{Properties of limits recalled}
\vspace{.5cm}  Recall that the coproduct of two
objects $a$ and $b$  in a category  $\sf C$ is an object $a \star b$ which is
endowed with two morphisms $\iota_{a} : a \longrightarrow a \star
b$ and $\iota_{b} : b \longrightarrow a \star b$ having the
following universal property:
\begin{quote}
For any object $c$ of $\sf C$ and
any pair of morphisms $f : a \longrightarrow c$ and $g : b
\longrightarrow c$, there is a unique morphism $h : a \star b
\longrightarrow c$, such that $h \circ \iota_{a} = f$ and $h \circ
\iota_{b} = g$. The morphism $h$ is usually denoted by $(f,\,g)$.\end{quote}
\para
We recall from \cite{EM} the following Lemma  which provides certain conditions which imply the
triviality of the inverse limit.

\begin{Lemma}\cite{EM}\label{lemmaem}
Let $\sf C$ be a small category and ${\mathfrak F} : \sf C
\longrightarrow \sf Ab$ a functor to the category of abelian
groups, and suppose that the following conditions are satisfied:
\begin{quote}
$(i)$ Any two objects $a,\,b$ of $\sf C$ have a coproduct
    $(a \star b ,\, \iota_{a} ,\, \iota_{b})$ as above.
\para\noindent
$(ii)$ For any two objects $a,\,b$ of $\sf C$ the morphisms
     $\iota_{a} : a \longrightarrow a \star b$ and\linebreak
     $\iota_{b} : b \longrightarrow a \star b$ induce
     a monomorphism
     \[ ({\mathfrak F}(\iota_{a}),\,{\mathfrak F}(\iota_{b})) :
        {\mathfrak F}(a) \oplus {\mathfrak F}(b) \longrightarrow
        {\mathfrak F}(a \star b) \]
     of abelian groups.\end{quote}
Then, the inverse limit $\ilimit {\mathfrak F}$ is the zero group.
\end{Lemma}

Indeed, let  $(x_{c})_{c\in Ob(\sf C)} \in \ilimit {\mathfrak F}$
be a compatible family and fix an object $a$ of $\sf C$. We
consider the coproduct $a \star a$ of two copies of $a$ and the
morphisms $\iota_{1} : a \longrightarrow a \star a$ and $\iota_{2}
: a \longrightarrow a \star a$. Then, we have
\[ {\mathfrak F}(\iota_{1})(x_{a}) = x_{a \star a} =
   {\mathfrak F}(\iota_{2})(x_{a}) \]
and hence the element $(x_{a},-x_{a})$ is contained in the kernel
of the additive map
\[ ({\mathfrak F}(\iota_{1}),{\mathfrak F}(\iota_{2})) :
   {\mathfrak F}(a) \oplus {\mathfrak F}(a) \longrightarrow
   {\mathfrak F}(a \star a) . \]
In view of our assumption, this latter map is injective and hence
$x_{a}=0$. Since this is the case for any object $a$ of $\sf C$,
we conclude that the compatible family $(x_{c})_{c\in Ob(\sf C)}$
vanishes, as asserted.

\para
We next mention  another  property of the limit functor.
 Recall that an object $c_0$ of a category $\sf C$ is called {\it quasi-initial} if the set
Hom$_{\sf C}(c_0,\,c)$ is non-empty for every  object $c$ of $\sf C$.
\para
\begin{Lemma}\cite{EP}\label{ep}
Let $\sf C$ be a category with a quasi-initial object $c_0$. Then,
for any functor ${\mathfrak F} : \sf C \longrightarrow \sf Ab$,
the natural map
$$
\ilimit \mathfrak F\to \mathfrak F(c_0)
$$
is injective, whereas its image consists of those elements $x\in
\mathfrak F(c_0)$ which equalize any pair of maps $\mathfrak
F(f_i): \mathfrak F(c_0)\to \mathfrak F(c),\ i=1,\,2$, where
$f_1,\,f_2\in Hom_{\sf C}(c_0,\,c)$ (i.e., $\mathfrak
F(f_1)(x)=\mathfrak F(f_2)(x)$).
\end{Lemma}
\para
Observe that given the category $\EuScript C$ and an  object $G\in {\sf Ob}(\EuScript C)$, the category
${\sf Fext}_\EuScript C(G)$ consists of quasi-initial objects.
\para
\section{Derived functors of certain functors on {\sf Ab}} \vspace{.5cm}
In this section we study the limits $\ilimit \mathfrak F$ for certain
functors $\mathfrak F\in {\sf Ab}^{{\sf Fext}_{\sf Ab}(A)}$ for
abelian groups $A$.
\para

To fix notation, let $\otimes^n:{\sf Ab}\to {\sf Ab}$, $n\geq 1$,  be the {\it $n$-th tensor power functor} $A\mapsto A^{\otimes
n}:=\underbrace{A\otimes \cdots \otimes A}_{n\ \text{terms}}$. The
symmetric group $\Sigma_n$ of degree $n$ acts naturally on $A^{\otimes n}:$
\begin{equation}\label{action}
\sigma (x_1\otimes \dots \otimes
x_n)=x_{\sigma(1)}\otimes\dots\otimes x_{\sigma(n)},\ x_i\in A,\
\sigma\in \Sigma_n.
\end{equation}
We thus have the {\it $n$-th symmetric power functor} $SP^n:{\sf Ab}\to {\sf Ab}$ with $SP^n(A)=A^{\otimes n}$ modulo the subgroup generated by the elements
$\sigma (x_1\otimes \dots \otimes
x_n)-x_{\sigma(1)}\otimes\dots\otimes x_{\sigma(n)},\ x_i\in A,\
\sigma\in \Sigma_n$. The {\it $n$-th exterior power functor}
$\Lambda^n: \sf Ab\to \sf Ab$ is
defined by $A\mapsto A^{\otimes n}$ modulo the subgroup generated by the elements $x_1\otimes \ldots\otimes x_n$ with $x_1,\,\ldots\,,\,x_n\in A$ and $x_i=x_{i+1} $ for some $i$.
The {\it $n$-th divided power functor} $\Gamma_n:\sf Ab\to \sf Ab$ is defined, for $A\in \sf Ab$,  to the $n$-th homogeneous component of  the graded  group
$\Gamma_\ast(A)$  generated by symbols $\gamma_i(x)$ of degree
$i\geq 0$ satisfying the following relations for all $x,\,y \in A$:
\begin{align*} & (i)\ \gamma_0(x) = 1 \\ & (ii)\ \gamma_1(x)=x\\ &
(iii)\
\gamma_s(x)\gamma_t(x)=\binom{s+t}{s}\gamma_{s+t}(x)\\
& (iv)\ \gamma_n(x+y)=\sum_{s+t=n}\gamma_s(x)\gamma_t(y),\ n\geq 1\\
& (v)\ \gamma_n(-x)=(-1)^n\gamma_n(x),\ n\geq 1.
\end{align*}
In particular, the canonical map $ A \to \Gamma_1(A)$ is an
isomorphism. It is known that, for a free abelian group $A$, there is a natural
isomorphism
$$
\Gamma_n(A)\simeq (A^{\otimes n})^{\Sigma_n},\ n\geq 1
$$
where the action of the symmetric group $\Sigma_n$ on $A^{\otimes
n}$ is defined as in (\ref{action}).

\para
Let $n\geq 0$ be an integer, and $T$ an endofunctor on the
category ${\sf Ab}$ of abelian groups. The doubly indexed family
$L_iT(-,\,n)$ of derived functors, in the sense of Dold-Puppe
\cite{DP}, of $T$
  are defined by
$$
L_iT(A,\,n)=\pi_iTN^{-1}P_\ast[n],\ i\geq 0,\ A\in {\sf Ab},
$$
where $P_\ast[n] \to A$ is a projective resolution of $A$ of level $n$, and
 $N^{-1}$ is  the Dold-Kan transform, which is the the inverse of the Moore normalization  functor
$$
N:  \mathcal{S}({\sf Ab}) \to \mathcal{C}h({\sf Ab})
$$
from the category of simplicial abelian groups to the category of chain complexes (see, for example,  \cite{MP}, pp.\,306, 326; or \cite{W}, Section 8.4). For any functor
$T$, we set $$L_iT(A):= L_iT(A,\,0),\ i\geq 0.$$

For abelian groups $B_1,\,\dots\,,\, B_n$, let the group
$\text{Tor}_i(B_1,\,\dots\,,\, B_n)$ denote the $i$-th homology
group of the complex $P_1\otimes \dots \otimes P_n,$ where $P_j$
is a $\mathbb Z$-flat resolution of $B_j$ for $j=1,\,\dots,\, n$.
We clearly have $$\text{Tor}_0(B_1,\,\dots,\, B_n)=B_0\otimes
\dots \otimes B_n,\ \text{Tor}_i(B_1,\,\dots,\, B_n)=0,\ i\geq
n.$$ It turns out from the Eilenberg-Zilber theorem that the
derived functors of the $n$-th tensor power can be described as
$$
L_i\otimes^n(A)=\text{Tor}_i(\underbrace{A, \cdots, A}_{n\
\text{terms}}),\ 0\leq i\leq n-1.
$$
We will use the following notation:
$$
\text{Tor}^{[n]}(A):=\text{Tor}_{n-1}(\underbrace{A, \cdots,
A}_{n\ \text{terms}}),\ n\geq 2.
$$

\para
\begin{Theorem}\label{tortheorem}
For $n\geq 2$, there is an isomorphism of abelian groups
$$
{\rm Tor}^{[n]}(A)\simeq \ilimit (F^{\otimes n}/H^{\otimes n}),\
A\in {\sf Ab},
$$
where the limit is taken over the category ${\sf Fext}_{\sf Ab}(A)$ of free
extensions $H\rightarrowtail F\twoheadrightarrow A$ in the category {\sf Ab}.
\end{Theorem}
\para
To proceed with the proof, we first recall the following result which is well-known.
\para
\begin{Lemma}\label{torlemma}
Let $A=F_1/H_1,\ B=F_2/H_2$, where $F_1,\,F_2$ are free abelian
groups. Then there is an isomorphism of abelian groups
$$
{\rm Tor}(A,\,B)=\frac{(H_1\otimes F_2)\cap (F_1\otimes
H_2)}{H_1\otimes H_2},
$$
where the intersection is taken in $F_1\otimes F_2$.
\end{Lemma}
\para
Indeed, the above result  follows directly from the exact sequence of abelian
groups
$$
0\to \text{Tor}(A,\,B)\to H_1\otimes B\to F_1\otimes B\to A\otimes B\to 0
$$
and isomorphisms $H_1\otimes B\simeq (H_1\otimes F_2)/(H_1\otimes
H_2),\ F_1\otimes B\simeq (F_1\otimes F_2)/(F_1\otimes H_2)$.

\para
\begin{Lemma}
Let $A=F/H$, where $F$ is a free abelian group. Then, for every
$n\geq 2$, there is an isomorphism of abelian groups
$$
{\rm Tor}^{[n]}(A)\simeq \frac{\bigcap_{i=1}^n (H^{\otimes
i-1}\otimes F\otimes H^{\otimes n-i})}{H^{\otimes n}}
$$
where the intersection is taken in $F^{\otimes n}$.
\end{Lemma}
\par
\begin{proof}
Observe that
\begin{equation}\label{Kunneth}
\text{Tor}^{[n]}(A)\simeq \text{Tor}(\text{Tor}^{[n-1]}(A),\,A),\
n\geq 3.
\end{equation}
To see (\ref{Kunneth}), one can apply the K\"unneth formula to
the tensor product of the chain complexes $P\otimes \dots \otimes
P\ (n-1\ \text{times})$ and $P$, where $P$ is a projective
resolution of $A$. The Lemma follows by inductive argument and
Lemma \ref{torlemma}.
\end{proof}

\para\noindent{\it Proof of Theorem \ref{tortheorem}.} Given $A=F/H,$
where $F$ is a free abelian group, consider the following exact
sequence of abelian groups:
\begin{equation}\label{seqab}
0\to \text{Tor}^{[n]}(A)\to F^{\otimes n}/H^{\otimes n}\to
F^{\otimes n}/\bigcap_{i=1}^n (H^{\otimes i-1}\otimes F\otimes
H^{\otimes n-i})\to 0.
\end{equation}
The sequence (\ref{seqab}) is natural in the following sense: any
morphism in  ${\sf Fext}_{\sf Ab}(A)$, say \linebreak $f:(F_1,\,\pi_1) \to (F_2,\,\pi_2)$
(with $H_1=\ker(\pi_1),\ H_2=\ker(\pi_2)$) implies the commutative
diagram
$$
\xyma{0 \ar@{->}[r]  & \text{Tor}^{[n]}(A) \ar@{->}[r] \ar@{=}[d]
& F_1^{\otimes n}/H_1^{\otimes n} \ar@{->}[r] \ar@{->}[d] &
F_1^{\otimes n}/\bigcap_{i=1}^n( H_1^{\otimes i-1}\otimes
F_1\otimes
H_1^{\otimes n-i}) \ar@{->}[d]\\
0 \ar@{->}[r] & \text{Tor}^{[n]}(A) \ar@{->}[r] & F_2^{\otimes
n}/H_2^{\otimes n} \ar@{->}[r] & F_2^{\otimes n}/\bigcap_{i=1}^n(
H_2^{\otimes i-1}\otimes F_2\otimes H_2^{\otimes n-i}) }
$$
Since the inverse limit functor is left exact in ${\sf Ab}^{{\sf
Fext}_{\sf Ab}(A)}$, we obtain a natural monomorphism
\begin{equation}\label{embed}
\text{Tor}^{[n]}(A)\hookrightarrow \ilimit F^{\otimes
n}/H^{\otimes n},\ n\geq 2.
\end{equation}
Given a free presentation $A=F/H$ in {\sf Ab}, consider the two morphisms in ${\sf
Fext}_{\sf Ab}(A)$
\begin{equation}\label{dofmap}
\xyma{0 \ar@{->}[r] & H \ar@{->}[r] \ar@{->}[d] & F \ar@{->}[r]
\ar@{->}[d]^{f_1,\,f_2} & A\ar@{->}[r] \ar@{=}[d]
& 0\\
0 \ar@{->}[r] & F\oplus H \ar@{->}[r] & F\oplus F \ar@{->}[r] &
A\ar@{->}[r] & 0} \end{equation} given by setting:
\begin{align*}
& f_1: g\mapsto (0,\,g),\ g\in F,\\
& f_2: g\mapsto (g,\,g),\ g\in F.
\end{align*}
Let $\alpha\in F^{\otimes n}/H^{\otimes n}$ be an element which
belongs to the equalizer of the maps
$$
f_1^*,\,f_2^*: F^{\otimes n}/H^{\otimes n}\to \frac{(F\oplus
F)^{\otimes n}}{(F\oplus H)^{\otimes n}}
$$
 induced by $f_1,\,f_2$ respectively. Express  $\alpha$ as a coset
$$
\alpha=(\sum_ig_1^{i}\otimes \dots \otimes g_n^i)+H^{\otimes n},\
g_j^i\in F.
$$
Identifying
$
\frac{(F\oplus F)^{\otimes n}}{(F\oplus H)^{\otimes
n}}$ with $\bigoplus_{(i_1,\,\dots,\, i_n)\in
\{0,\,1\}^n}\frac{F^{\otimes n}}{ C_{i_1}\otimes \dots\otimes
C_{i_n}}
$
where $C_0=F$ and $ C_1=H$, we can describe $f_i^*(\alpha),\ i=1,\,2$, as
\begin{align*}
& f_1^*(\alpha)=(0,\,\dots\,,\,0,\,\sum_ig_1^{i}\otimes \dots
\otimes
g_n^i)\\
& f_2^*(\alpha)=(\sum_ig_1^{i}\otimes \dots \otimes
g_n^i,\,\dots\,,\,\sum_ig_1^{i}\otimes \dots \otimes g_n^i)
\end{align*}
Since $\alpha$ lies in the equalizer of $f_1^*$ and $f_2^*$, we
conclude that
$$
\sum_ig_1^{i}\otimes \dots \otimes g_n^i\in \bigcap_{i=1}^n
(H^{\otimes i-1}\otimes F\otimes H^{\otimes n-i}).
$$
The category ${\sf Fext}_{\sf Ab}(A)$ clearly consists of quasi-initial
objects; hence  Lemma \ref{ep} implies that the natural map
(\ref{embed}) is an isomorphism and the theorem is proved.\ $\Box$
\para

\begin{Theorem}
For every abelian group $A$ and integer $n\geq 2,$  there are natural isomorphisms
\begin{align}
& L_{n-1}SP^n(A)\simeq \ilimit \Lambda^n(F)/\Lambda^n(H)\\
& L_{n-1}\Lambda^n(A)\simeq \ilimit \Gamma_n(F)/\Gamma_n(H)
\end{align}
where the limits are taken over the category ${\sf Fext_{Ab}}(A)$
of free extensions $H\rightarrowtail F\twoheadrightarrow A$.
\end{Theorem}
\begin{proof}
Given a free extension $$ 0\to H\buildrel{f}\over\to F\to A\to 0
$$
the Koszul complexes
\begin{equation}\label{kos1}
 0\to \Lambda^n(H)\buildrel{\kappa_n}\over\to
\Lambda^{n-1}(H)\otimes F\buildrel{\kappa_{n-1}}\over\to \dots
\buildrel{\kappa_2}\over\to H\otimes
SP^{n-1}(F)\buildrel{\kappa_1}\over\to SP^n(F) \end{equation}
and
\begin{equation}
 0\to \Gamma_n(H)\buildrel{\kappa^n}\over\to
\Gamma_{n-1}(H)\otimes F\buildrel{\kappa^{n-1}}\over\to \dots
\buildrel{\kappa^2}\over\to H\otimes
\Lambda^{n-1}(F)\buildrel{\kappa^1}\over\to
\Lambda^n(F)\label{kos2}
\end{equation}
represent models of the objects $LSP^n(A)$ and $L\Lambda^n(A)$ in
the derived category (see \cite{Kock}, Proposition 2.4 and Remark 2.7). In these complexes, the maps
\begin{align*}
& \kappa_{k+1}: \Lambda^{k+1}(H)\otimes SP^{n-k-1}(F)\to
\Lambda^k(H)\otimes SP^{n-k}(F),\ k=0,\dots, n-1\\
& \kappa^{k+1}: \Gamma_{k+1}(H)\otimes \Lambda^{n-k-1}(F)\to
\Gamma_k(H)\otimes \Lambda^{n-k}(F),\ k=0,\,\dots\,,\, n-1
\end{align*}
are defined by setting:
\begin{multline*}
\kappa_{k+1}: p_1\wedge \dots\wedge p_{k+1}\otimes q_{k+2}\dots
q_n\mapsto \\ \sum_{i=1}^{k+1}(-1)^{k+1-i}p_1\wedge \dots \wedge
\hat p_i\wedge \dots \wedge p_{k+1}\otimes f(p_i)\,q_{k+2}\dots
q_n\\
p_1,\,\dots\,, \,p_{k+1}\in H,\ q_{k+2},\,\dots\,, \,q_n\in F.
\end{multline*}
and
\begin{multline*}
\kappa^{k+1}: \gamma_{r_1}(p_1)\ldots  \gamma_{r_k}(p_k) \otimes
q_1\wedge\ldots \wedge q_{n-k-1}
\mapsto \\
 \sum_{j=1}^k  \gamma_{r_1}(p_1)\ldots \gamma_{r_{j-1}}(p_j)\ldots \gamma_{r_k}(p_k)
\otimes f(p_j) \wedge q_1\wedge \ldots \wedge q_{n-k-1},\
p_1,\,\dots\,,\, p_k\in H,\ q_1,\,\dots\,, \,q_{n-k-1}\in F
\end{multline*}
In particular, the homology groups of complexes (\ref{kos1}) and
(\ref{kos2}) are isomorphic to the derived functors $L_iSP^n(A)$
and $L_i\Lambda^n(A)$ respectively. If $f: H\to F$ is the identity
map and $A=0$, the complexes (\ref{kos1}) and (\ref{kos2}) are
acyclic complexes. The commutative diagram
$$
\xyma{H \ar@{^{(}->}[r]^f \ar@{^{(}->}[d]^f & F\ar@{=}[d]\\
F \ar@{=}[r] & F}
$$
implies the following diagrams with exact columns:
\begin{equation}\label{d1}
\xyma{\Lambda^n(H) \ar@{^{(}->}[d] \ar@{^{(}->}[r]^{\kappa_n\ \ \
\ \ } & \Lambda^{n-1}(H)\otimes F \ar@{^{(}->}[d] \ar@{->}[r]^{\ \
\ \ \ \ \kappa_{n-1}} & \dots \ar@{->}[r]^{\kappa_2\ \ \ \ \ \ } &
H\otimes SP^{n-1}(F)\ar@{^{(}->}[d]
\ar@{->}[r]^{\ \ \ \kappa_1} & SP^n(F)\ar@{=}[d]\\
\Lambda^n(F) \ar@{->>}[d] \ar@{^{(}->}[r] &
\Lambda^{n-1}(F)\otimes F\ar@{->}[r] \ar@{->>}[d] &
\dots \ar@{->}[r] & F\otimes SP^{n-1}(F) \ar@{->>}[d] \ar@{->}[r] & SP^n(F)\\
\frac{\Lambda^n(F)}{\Lambda^n(H)} \ar@{->}[r] &
\frac{\Lambda^{n-1}(F)}{\Lambda^{n-1}(H)}\otimes F \ar@{->}[r]
\ar@{->}[r] & \dots \ar@{->}[r] & A\otimes SP^{n-1}(F),}
\end{equation}
\begin{equation}\label{d2}
\xyma{\Gamma_n(H) \ar@{^{(}->}[d] \ar@{^{(}->}[r]^{\kappa^n\ \ \ \
\ } & \Gamma_{n-1}(H)\otimes F \ar@{^{(}->}[d] \ar@{->}[r]^{\ \ \
\ \ \ \kappa^{n-1}} & \dots \ar@{->}[r]^{\kappa^2\ \ \ \ \ \ } &
H\otimes \Lambda^{n-1}(F)\ar@{^{(}->}[d]
\ar@{->}[r]^{\ \ \ \kappa^1} & \Lambda^n(F)\ar@{=}[d]\\
\Gamma_n(F) \ar@{->>}[d] \ar@{^{(}->}[r] & \Gamma_{n-1}(F)\otimes
F\ar@{->}[r] \ar@{->>}[d] &
\dots \ar@{->}[r] & F\otimes \Lambda^{n-1}(F) \ar@{->>}[d] \ar@{->}[r] & \Lambda^n(F)\\
\frac{\Gamma_n(F)}{\Gamma_n(H)} \ar@{->}[r] &
\frac{\Gamma_{n-1}(F)}{\Gamma_{n-1}(H)}\otimes F \ar@{->}[r]
\ar@{->}[r] & \dots \ar@{->}[r] & A\otimes \Lambda^{n-1}(F).}
\end{equation}
Since the middle horizontal sequences in the diagrams (\ref{d1})
and (\ref{d2}) are exact, we obtain the following exact sequences:
\begin{align}
& 0\to L_{n-1}SP^n(A) \to \frac{\Lambda^n(F)}{\Lambda^n(H)}\to
\frac{\Lambda^{n-1}(F)}{\Lambda^{n-1}(H)}\otimes F\label{zz1}\\
& 0\to L_{n-1}\Lambda^n(A)\to \frac{\Gamma_n(F)}{\Gamma_n(H)}\to
\frac{\Gamma_{n-1}(F)}{\Gamma_{n-1}(H)}\otimes F\label{zz2}
\end{align}
Since the inverse limit functor is left exact, we obtain the
following natural sequences:
\begin{align}
& 0\to L_{n-1}SP^n(A) \to \ilimit
\frac{\Lambda^n(F)}{\Lambda^n(H)}\to
\ilimit\frac{\Lambda^{n-1}(F)}{\Lambda^{n-1}(H)}\otimes F\label{ss1}\\
& 0\to L_{n-1}\Lambda^n(A)\to
\ilimit\frac{\Gamma_n(F)}{\Gamma_n(H)}\to \ilimit
\frac{\Gamma_{n-1}(F)}{\Gamma_{n-1}(H)}\otimes F\label{ss2}
\end{align}
where the limits are taken, as usual, over the category of extensions
$H\rightarrowtail F\twoheadrightarrow A$. We claim that for $n\geq
2,$ and every $k\geq 1$,
\begin{equation}\label{es3}
\ilimit \left(\frac{\Lambda^{n-1}(F)}{\Lambda^{n-1}(H)}\otimes
F^{\otimes k}\right)=\ilimit
\left(\frac{\Gamma_{n-1}(F)}{\Gamma_{n-1}(H)}\otimes F^{\otimes
k}\right)=0.
\end{equation}
For $n=2,$ the functor
$$
\{H\rightarrowtail F\twoheadrightarrow A\}\mapsto A\otimes
F^{\otimes k}
$$
satisfies the conditions of Lemma \ref{lemmaem} and the assertion
follows. Now assume that (\ref{es3}) is proved for a fixed $n\geq
1$. Consider the tensor products of sequences (\ref{zz1}) (for
$n+1$) and (\ref{zz2}) with $F^{\otimes k}$:
\begin{align*}
& 0\to L_{n}SP^{n+1}(A)\otimes F^{\otimes k}\to
\frac{\Lambda^{n+1}(F)}{\Lambda^{n+1}(H)}\otimes F^{\otimes k}\to
\frac{\Lambda^{n}(F)}{\Lambda^{n}(H)}\otimes F^{\otimes k+1}\\
& 0\to L_{n}\Lambda^{n+1}(A)\otimes F^{\otimes k}\to
\frac{\Gamma_{n+1}(F)}{\Gamma_{n+1}(H)}\otimes F^{\otimes k}\to
\frac{\Gamma_{n}(F)}{\Gamma_{n}(H)}\otimes F^{\otimes k+1}
\end{align*}
By induction,
$$
\ilimit \left(\frac{\Lambda^{n}(F)}{\Lambda^{n}(H)}\otimes
F^{\otimes k+1}\right)=\ilimit
\left(\frac{\Gamma_{n}(F)}{\Gamma_{n}(H)}\otimes F^{\otimes
k+1}\right)=0
$$
and the functors
\begin{align*}
& \{H\rightarrowtail F\twoheadrightarrow A\}\mapsto
L_nSP^{n+1}(A)\otimes F^{\otimes k}\\
& \{H\rightarrowtail F\twoheadrightarrow A\}\mapsto
L_n\Lambda^{n+1}(A)\otimes F^{\otimes k}
\end{align*}
satisfy the conditions of Lemma \ref{lemmaem}. Hence (\ref{es3})
follows. The statement of the theorem now follows from sequences
(\ref{ss1}) and (\ref{ss2}).
\end{proof}
 \vspace{.5cm}
\section{Higher limits}
\vspace{.5cm}

Let $\sf C$ be a small category. The first derived functor
$$
 \ilimit^1:  {\sf Gr}^{\sf C}\to \text{pointed sets}
$$
can be defined via cosimplicial replacement in the category ${\sf
Gr}^{\sf C}$ described in \cite{BK}. Given $\mathfrak F\in {\sf
Gr}^{\sf C}$, define a cosimplicial replacement
${\prod}^*\mathfrak F$, a cosimplicial group, with
$$
{\prod}^n\mathfrak F=\prod_{u\in I_n}\mathfrak F(i_0),\
u=\{i_0\buildrel{\alpha_1}\over\leftarrow \dots
\buildrel{\alpha_n}\over\leftarrow i_n\}
$$
and coface and codegeneracy maps induced by
\begin{align*}
& d^0: \mathfrak F(i_1)\buildrel{\mathfrak F(\alpha_1)}\over\to
\mathfrak F(i_0),\\
& d^j: \mathfrak F(i_0)\buildrel{id}\over\to \mathfrak F(i_0),\
0<j\leq n,\\
& s^j: \mathfrak F(i_0)\buildrel{id}\over\to \mathfrak F(i_0),\
0\leq j\leq n.
\end{align*}
One can check that there is a natural isomorphism (see \cite{BK})
$$
\ilimit \mathfrak F=\pi^0{\prod}^*\mathfrak F.$$ The derived
functor of the inverse limit can be defined as
$$
\ilimit^1 \mathfrak F=\pi^1{\prod}^*\mathfrak F\in \text{pointing
sets}. $$We then have the following:
\para
\begin{prop}\label{lemmainverse}
Let $\sf 1$ be an identity functor in ${\sf Gr}^{\sf C}$ and let
$$
\sf 1\to \mathfrak F_1\to \mathfrak F_2\to \mathfrak F_3\to \sf 1
$$
be a short exact sequence in ${\sf Gr}^{\sf C}$. There is a
natural long exact sequence of groups and pointed spaces:
\begin{equation}\label{a09}
1\to \ilimit \mathfrak F_1\to \ilimit \mathfrak F_2\to \ilimit
\mathfrak F_3\to \ilimit^1 \mathfrak F_1\to \ilimit^1 \mathfrak
F_2\to \ilimit^1 \mathfrak F_3.
\end{equation}
\end{prop}\para
In the case of the category ${\sf Ab}^{\sf C}$, the functor
$\ilimit^1$ has values in $\sf Ab$ and the sequence (\ref{a09}) is
a long exact sequence of abelian groups. In this case there is a
cochain complex of abelian groups defined by
$$
{\prod}^\bullet \mathfrak F: {\prod}^0\mathfrak
F\buildrel{\delta^0}\over\rightarrow {\prod}^1\mathfrak
F\buildrel{\delta^1}\over\rightarrow\dots
$$
with
\begin{multline*}
\delta^n(a^n)\{i_0\buildrel{\alpha_1}\over\leftarrow \dots
\buildrel{\alpha_{n+1}}\over\leftarrow i_{n+1}\}=\\ \mathfrak
F(i_0\buildrel{\alpha_1}\over\leftarrow
i_1)a^n\{i_1\buildrel{\alpha_1}\over\leftarrow \dots
\buildrel{\alpha_{n+1}}\over\leftarrow
i_{n+1}\}+\sum_{j=1}^{n+1}(-1)^ja^n\{i_0\buildrel{\alpha_1}\over\leftarrow
\dots \leftarrow \hat i_j\leftarrow\dots
\buildrel{\alpha_{n+1}}\over\leftarrow i_{n+1}\},\\ a^n\in
{\prod}^n\mathfrak F,
\end{multline*}
such that the derived functors of the inverse limit are the
cohomology groups:
$$
\ilimit^n \mathfrak F=H^n({\prod}^\bullet \mathfrak F),\ n\geq 0
$$
(see \cite{Jensen}, Theorem 4.1). Clearly,
$$
\ilimit \mathfrak F=\ilimit^0 \mathfrak F=\ker(\delta^0).
$$

The question of vanishing of higher limits of functors defined on
small categories in general reduces to the computation of local
cohomology of nerves of these categories. We give a simple
condition for the vanishing of $\ilimit^1$.

\para
\begin{prop}
Let $\sf C$ be a category with a quasi-initial object and
${\mathfrak F} : \sf C \longrightarrow \sf Ab$ a functor. Suppose
that every pair of morphisms in $\sf C$ has a coequalizer, i.e.,
for every pair of morphisms $\varepsilon_1,\,\varepsilon_2: I_1\to
I_0,\ I_1,\,I_0\in {\sf Ob}(\sf C)$ there is a morphism $\varepsilon:
I_0\to I(I_0,\,I_1)$ in $\sf C$ such that the following diagram is
commutative
\begin{equation*}
\xymatrix{I(I_0,\,I_1)& &
I_0\ar[ll]_{\varepsilon}&I_1\ar@/_30pt/[lll]_{\varepsilon\circ\varepsilon_1}^{\;}="b"\ar@/_10pt/[l]^{\varepsilon_1}^<(1){\;}="c"
\ar@/^30pt/[lll]^{\varepsilon\circ\varepsilon_2}_{\;}="d"\ar@/^10pt/[l]_{\varepsilon_2}},
\end{equation*}
i.e.,
$\varepsilon\circ\varepsilon_1=\varepsilon\circ\varepsilon_2$ and
the induced map $\mathfrak F(\varepsilon): \mathfrak F(I_0)\to
\mathfrak F(I(I_0,I_1))$ is injective. Then $\ilimit^1\mathfrak
F=0$.
\end{prop}

\noindent{\it Proof.} Let $a^1\in {\prod}^1\mathfrak
F=\prod_{i_0\leftarrow i_1}^1\mathfrak F$ be a 1-cocycle, i.e.,
\begin{equation}\label{cocc}
\delta^1a^1(i_0\leftarrow i_1\leftarrow i_2)=\mathfrak
F(i_0\leftarrow i_1)a^1(i_1\leftarrow i_2)+a^1(i_0\leftarrow
i_1)-a^1(i_0\leftarrow i_2)=0
\end{equation}
for every diagram $i_0\leftarrow i_1\leftarrow i_2$. Given two
morphisms $\varepsilon_1, \,\varepsilon_2: i_1\to i_0$, consider a
morphism $\varepsilon: i_0\to I(i_0,\,i_1),$ such that $\mathfrak
F(\varepsilon): \mathfrak F(i_0) \to \mathfrak F(I(i_0,\,i_1))$ is a
monomorphism of abelian groups and
$\varepsilon\circ\varepsilon_1=\varepsilon\circ\varepsilon_2$. The
cocycle condition (\ref{cocc}) implies that
\begin{align*}
\mathfrak
F(\varepsilon)a^1(i_0\buildrel{\varepsilon_1}\over\leftarrow
i_1)=a^1(I(i_0,\,i_1)\buildrel{\varepsilon\circ\varepsilon_1}\over\leftarrow i_1)-a^1(I(i_0,\,i_1)\buildrel{\varepsilon}\over\leftarrow i_0)\\
\mathfrak
F(\varepsilon)a^1(i_0\buildrel{\varepsilon_2}\over\leftarrow
i_1)=a^1(I(i_0,\,i_1)\buildrel{\varepsilon\circ\varepsilon_2}\over\leftarrow
i_1)-a^1(I(i_0,\,i_1)\buildrel{\varepsilon}\over\leftarrow i_0)
\end{align*}
and, therefore,
$$
\mathfrak
F(\varepsilon)a^1(i_0\buildrel{\varepsilon_1}\over\leftarrow
i_1)=\mathfrak
F(\varepsilon)a^1(i_0\buildrel{\varepsilon_2}\over\leftarrow i_1)
$$
in $\mathfrak F(I(i_0,\,i_1))$. Since $\mathfrak F(\varepsilon)$ is
a monomorphism, we conclude that
\begin{equation}\label{cocc1}
a^1(i_0\buildrel{\varepsilon_1}\over\leftarrow
i_1)=a^1(i_0\buildrel{\varepsilon_2}\over\leftarrow i_1)
\end{equation}
in $\mathfrak F(i_0)$. Now we can take a quasi-initial object
$i\in {\sf Ob}(\sf C)$ and define an element $a^0\in {\prod}^0\mathfrak
F$ by setting
$$
a^0(i_0)=a^1(i_0\leftarrow i)
$$
for arbitrary map $i_0\leftarrow i$ (such a map exists, since $i$
is a quasi-initial object). The equality (\ref{cocc1}) implies
that this is a well-defined element. By definition, we have
$$
-a^1(i_0\leftarrow i_1)=\mathfrak F(i_0\leftarrow
i_1)a^0(i_1)-a^0(i_0)=\delta^0a^0(i_0\leftarrow i_1),
$$and the proof is complete.\ $\Box$
\para
At the moment we are not able to compute higher limits over
categories of free extensions. We present here an approach towards this problem and illustrate it with an application.  In particular,  we show that higher limits  `cover' certain homology functors.
\para
Given a category $\EuScript C$, object $G\in {\sf Ob}(\EuScript C)$, and the category of
free extensions ${\sf Fext}_{\EuScript C}(G),$ suppose we have two  pairs of
functors
\begin{align*}
& \mathcal H_1,\, \mathcal H_2: \EuScript C\to \sf Ab \\
& \mathcal F_1,\, \mathcal F_2: {\sf Fext}(G)\to \sf Ab
\end{align*}
such that, for every $\alpha\in {\sf Fext}_\EuScript C(G),$ there is a natural
4-term exact sequence
\begin{equation}\label{higher1}
0\to \mathcal H_2(G)\to \mathcal F_1(\alpha)\to \mathcal
F_2(\alpha)\to \mathcal H_1(G)\to 0
\end{equation}
which is natural in the sense that every morphism $\beta\to \alpha$
in ${\sf Fext}_\EuScript C(G)$ induces the commutative diagram
$$
\xyma{\mathcal H_2(G) \ar@{^{(}->}[r] \ar@{=}[d] & \mathcal
F_1(\beta) \ar@{->}[r] \ar@{->}[d] & \mathcal F_2(\beta)
\ar@{->>}[r] \ar@{->}[d] & \mathcal H_1(G) \ar@{=}[d]\\ \mathcal
H_2(G) \ar@{^{(}->}[r] & \mathcal F_1(\alpha) \ar@{->}[r] &
\mathcal F_2(\alpha) \ar@{->>}[r] & \mathcal H_1(G)}.
$$
Suppose further that
\begin{equation}\label{va}
\ilimit \mathcal F_2=0.
\end{equation}

\para
The condition (\ref{va}) implies the following exact sequences of
abelian groups: \begin{equation}\label{poi} \xyma{ & \ilimit^1 \mathcal F_1 \ar@{->}[d] \\
\mathcal H_1(G) \ar@{->}[rd]^f\ar@{^{(}->}[r] & \ilimit_{\alpha\in
{\sf Fext}_\EuScript C(G)}^1 C(\alpha) \ar@{->}[r] \ar@{->}[d] &
\ilimit^1\mathcal F_2 \ar@{->}[r] &\ilimit^1 \mathcal H_1(G)\\ &
\ilimit^2\mathcal H_2(G)}
\end{equation}
where $C(\alpha)=\operatorname{coker}\{\mathcal H_2(G)\to \mathcal
F_1(\alpha)\}=\ker\{\mathcal F_2(\alpha)\to \mathcal H_1(G)\},\
\alpha\in {\sf Fext}_{\EuScript C}(G).$
\para
For every $\alpha\in {\sf Fext}_\EuScript C(G)$, fix sections $s_\alpha:
\mathcal H_1(G)\to \mathcal F_2(\alpha)$ and $t_\alpha:
C(\alpha)\to \mathcal F_1(\alpha).$
To describe the map $f$, let $a\in \mathcal H_1(G)$ and let
$\gamma\to \beta\to \alpha$ be a diagram in ${\sf Fext}_{\EuScript C}(G)$ . Consider
the following diagram
$$
\xyma{\mathcal H_2(G) \ar@{^{(}->}[r] \ar@{=}[d] & \mathcal
F_1(\gamma) \ar@{->}[r] \ar@{->}[d]^{\mathcal F_1(\gamma\to
\beta)} & \mathcal F_2(\gamma) \ar@{->>}[r] \ar@{->}[d]^{\mathcal
F_2(\gamma\to \beta)} & \mathcal H_1(G) \ar@{=}[d]\\ \mathcal
H_2(G) \ar@{^{(}->}[r] \ar@{=}[d] & \mathcal F_1(\beta)
\ar@{->}[r] \ar@{->}[d]^{\mathcal F_1(\beta\to \alpha)}& \mathcal
F_2(\beta) \ar@{->>}[r] \ar@{->}[d]^{\mathcal F_2(\beta\to \alpha)} & \mathcal H_1(G)\ar@{=}[d]\\
\mathcal H_2(G) \ar@{^{(}->}[r] & \mathcal F_1(\alpha) \ar@{->}[r]
& \mathcal F_2(\alpha) \ar@{->>}[r] & \mathcal H_1(G)}
$$
Define
\begin{multline*}
a^2(\gamma\to \beta\to \alpha):=\mathcal F_1(\beta\to
\alpha)t_\beta(\mathcal F_2(\gamma\to
\beta)s_\gamma(a)-s_\beta(a))-\\ t_\gamma\mathcal F_2(\beta\to
\alpha)(\mathcal F_2(\gamma\to \beta)s_\gamma(a)-s_\beta(a))
\end{multline*}
The 2-cocycle condition can be checked directly; moreover, note that the element $\xi \in
\ilimit^2\mathcal H_2(G)$ defined by the cocycle $a^2(\gamma\to \beta\to \alpha)$ does not depend
on the choice of sections $s_\alpha,\,t_\alpha$. The map $f:\mathcal H_1(G)\to \ilimit^2\mathcal H_2(G)$ is thus the one given by $a\mapsto \xi$.
\para
We are interested in finding
conditions which imply the triviality of the map $f$.
\para
\begin{prop}\label{yon}
Suppose we have functors
$$
\mathcal F_3,\,\mathcal F_4,\,\mathcal F:\ {\sf Fext}_\EuScript C(G)\to {\sf Gr}
$$
such that the following conditions are satisfied:\\ $(1)$ There is a
natural diagram
\begin{equation}\label{aq}
\xyma{\mathcal H_2(G) \ar@{=}[d] \ar@{^{(}->}[r] & \mathcal
F_1(\alpha) \ar@{=}[d] \ar@{->}[r] & \mathcal F_3(\alpha)
\ar@{->>}[d] \ar@{->>}[r] & \mathcal F_4(\alpha)\ar@{->>}[d]\\
\mathcal H_2(G) \ar@{^{(}->}[r] & \mathcal F_1(\alpha) \ar@{->}[r]
& \mathcal F_2(\alpha) \ar@{->>}[r] & \mathcal H_1(G)}.
\end{equation} $(2)$ The natural map
$$
\ilimit \mathcal F_4\to \mathcal H_1(G)
$$
is an epimorphism.\\ $(3)$ For every $\alpha\in {\sf Fext}_\EuScript C(G)$,
there is a natural monomorphism
$$
\mathcal F_1(\alpha)\to \mathcal F(\alpha)
$$
and natural short exact sequences
\begin{align*}
& 1\to \mathcal F_1(\alpha)\to \mathcal F(\alpha)\to \mathcal
F_4(\alpha)\to 1\\ & 1\to \mathcal H_2(G)\to \mathcal F(\alpha)\to
\mathcal F_3(\alpha)\to 1.
\end{align*}
Then the natural map
$$
\ilimit \mathcal F_4\to \ilimit^2 \mathcal H_2(G)
$$
is the trivial map and therefore the map
$$
f: \mathcal H_1(G)\to \ilimit^2\mathcal H_2(G)
$$
is the zero map.
\end{prop}
\begin{proof} The proof follows from the functoriality of the
considered constructions and the following natural commutative
diagram:
$$
\xyma{\mathcal H_2(G) \ar@{=}[r] \ar@{=}[d] & \mathcal H_2(G)
\ar@{^{(}->}[d] \ar@{->}[r]^0 & \mathcal F_4(\alpha)
 \ar@{^{(}->}[d] \ar@{=}[r] & \mathcal F_4(\alpha) \ar@{=}[d]\\
\mathcal H_2(G) \ar@{->}[r] & \mathcal F(\alpha) \ar@{->}[r] &
\mathcal F_3(\alpha)\oplus \mathcal F_4(\alpha) \ar@{->}[r]^{\ \ \
\ \ \ \ \ (0,\,id)}
& \mathcal F_4(\alpha)\\
\mathcal H_2(G) \ar@{=}[u] \ar@{^{(}->}[r] & \mathcal F_1(\alpha)
\ar@{^{(}->}[u] \ar@{->}[r] & \mathcal F_3(\alpha) \ar@{^{(}->}[u]
\ar@{->>}[r] & \mathcal F_4(\alpha) \ar@{=}[u]}.
$$\end{proof}

\par We next give examples of functors satisfying (\ref{higher1}).
\para \noindent {\bf Examples.}\\

\noindent 1. Let $G$ be a group, $n\geq 1$, $\{R\rightarrowtail
F\twoheadrightarrow G\}\in {\sf Fext}_{\sf Gr}(G)$, then  there is
a natural  exact
sequence of abelian groups (see \cite{EM}):
\begin{equation}\label{groupex}
0\to H_{2n}(G)\to H_0(F,\,
R_{ab}^{\otimes n})\to H_1(F,\,R_{ab}^{\otimes n-1})\to
H_{2n-1}(G)\to 0.
\end{equation}

\vspace{.25cm} \noindent 2. Let $A$ be an associative algebra over
$\mathbb Q$,\, $n\geq 1$, $\{ I\rightarrowtail R\twoheadrightarrow
A\}\in {\sf Fext}_{{\sf Ass}_\mathbb Q}(A).$ There is a natural
exact sequence (see \cite{Q}):
$$
0\to HC_{2n}(A)\to HH_0(R/I^{n+1})\to H_1(R,\,R/I^n)\to
HC_{2n-1}(A)\to 0.
$$
 \para
Lemma \ref{lemmaem}
implies that, for the example 1 above,  the condition (\ref{va}) is satisfied for the functor
$$
\{R\rightarrowtail F\twoheadrightarrow G\}\mapsto H_1(F,\, R_{ab}^{\otimes n-1})
$$
for $n\geq 1$ (see \cite{EM} for details), hence there is an
isomorphism $H_{2n}(G)\simeq \ilimit H_0(F,\,R_{ab}^{\otimes n}).$
For the simplest case, namely for  $n=1$, functors from the exact sequence
(\ref{groupex}), the diagram (\ref{aq}) can be chosen to be the one given below (with the obvious maps):
$$ \xyma{H_2(G) \ar@{=}[d] \ar@{^{(}->}[r] & R/[R,\,F] \ar@{=}[d]
\ar@{->}[r] & F/(R\cap [F,\,F])
\ar@{->>}[d] \ar@{->>}[r] & G\ar@{->>}[d]\\
H_2(G) \ar@{^{(}->}[r] & R/[F,\,R] \ar@{->}[r] & F/[F,\,F]
\ar@{->>}[r] & H_1(G)}
$$
Proposition \ref{yon} then implies that the natural map $H_1(G)\to
\ilimit^2 H_2(G)$ is the zero map. Consequently, the diagram (\ref{poi})
implies that $H_1(G)$ is contained in a group, which is an
epimorphic image of $\ilimit^1(R/[F,\,R])$. We have thus proved the following:

\begin{Theorem}
If $G$ is not a perfect group, then $\ilimit^1(R/[F,\,R])$ is non-trivial.
\end{Theorem}

\section{Concluding remarks and questions} \vspace{.5cm} Observe that given an
object $G\in {\sf Ob}(\EuScript C),$ one can consider the category ${\sf
Fext}_2(G)$ of double (resp. triple etc) presentations of $G$. For
 simplicity, let us assume that we work in the category of
groups. The objects of ${\sf Fext}_2(G)$ are triples
$(F,\,R_1,\,R_2),$ where $F$ is a group, $R_1,\, R_2$ normal subgroups
in $F$, such that $F/R_1R_2=G$. The morphisms in ${\sf Fext}_2(G)$
are the diagrams of the form
$${\tiny
\xyma{ & R_1\cap R_2  \ar@{->}[rr]
\ar@{->}[ld] \ar@{-}[d] && R_2 \ar@{->}[dd] \ar@{->}[ld]\\
R_1 \ar@{->}[dd] \ar@{->}[rr] & \ar@{->}[d] & F \ar@{->}[dd]\\
& R_1'\cap R_2' \ar@{-}[r] \ar@{->}[ld] & \ar@{->}[r] & R_2' \ar@{->}[ld]\\
R_1' \ar@{->}[rr] & & F' } }
$$
which induce the identity isomorphism $F/R_1R_2\to F'/R_1'R_2'.$
It would be of interest to examine limits of functors over the category ${\sf Fext}_2(G)$.
For example, note that, given a group $G$, there is a
natural homomorphism
$$
H_3(G)\to \ilimit_{(F,\,R_1,\,R_2)\in {\sf Fext}_2(G)}\frac{R_1\cap
R_2}{[R_1,\,R_2][F,\,R_1\cap R_2]};$$for the construction of this map see
 the homology exact sequence in \cite{BG}.

\para
In the same way, one can make variations of Quillen's description (\ref{func22})
of cyclic homology. Given an associative algebra
$A$ over $\mathbb Q$, consider the category ${\sf Fext}_2(A)$ whose
objects  are the triples $(R,\,I_1,\,I_2)$, where $R$ is a
free algebra $I_1,\,I_2$ are ideals in $R$ and $R/(I_1+I_2)=A$. The
description (\ref{func22}) implies, for example, that for $n\geq
2,\ n>k\geq 1,$ there is a natural morphism
$$
\ilimit_{(R,\,I_1,\,I_2)\in {\sf
Fext}_2(A)}\frac{R}{I_1^{n+1-k}I_2^k+I_1^{k}I_2^{n+1-k}+[R,R]}\to
HC_{2n}(A)
$$and its investigation may be of interest for cyclic homolgy.
\para
Finally, one now knows how to define
even-dimensional homology of groups, Lie algebras, cyclic homology
of associative algebras as limits of certain functors over the
categories of  extensions. What can one say about higher limits
of the functors yielding this relationship?
\para
\section{Acknowledgement}
The authors would like to thank Ioannis Emmanouil and Gosha
Sharygin for useful discussions and important suggestions.

\newpage
\par\vspace{.5cm}\noindent
Roman Mikhailov\\Steklov Mathematical Institute\\
Department of Algebra\\
Gubkina 8\\
Moscow 119991\\
Russia\\
email: romanvm@mi.ras.ru
\par\vspace{.5cm}\noindent
Inder Bir S. Passi\\
INSA Senior Scientist\\
Centre for Advanced Study in Mathematics\\
Panjab University \\
Chandigarh 160014\\ India
\par\noindent
and
\par\noindent
Indian Institute of Science Education and Research  Mohali\\
MGSIPA Complex, Sector 19\\
Chandigarh 160019\\
India\\
email: ibspassi@yahoo.co.in

\end{document}